\documentclass[12pt, a4paper]{amsart}
\usepackage{amsmath}
\usepackage{amssymb}
\usepackage{amsthm}
\usepackage{url}
\usepackage[top=25.4mm, bottom=25.4mm, left=31.7mm, right=32.2mm]{geometry}
\theoremstyle{plain}
\newtheorem{thm}{Theorem}[section]
\newtheorem{lem}[thm]{Lemma}
\newtheorem{prop}[thm]{Proposition}
\newtheorem{cor}[thm]{Corollary}
\newtheorem{rem}[thm]{Remark}
\newtheorem{defn}[thm]{Definition}
\begin{document}
\title{Chaos and Entropy for Interval Maps }
\author[J. Li]{Jian Li}
\date{\today}
\address{Department of Mathematics, University of Science and Technology of China,
Hefei, Anhui, 230026, P.R. China}
\email{lijian09@mail.ustc.edu.cn}
\begin{abstract}
In this paper,  various chaotic properties and their relationships
for interval maps are discussed. It is shown that the proximal
relation is an equivalence relation for any zero entropy interval
map. The structure of the set of $f$-nonseparable pairs is well
demonstrated and so is its relationship to Li-Yorke chaos. For a
zero entropy interval map, it is shown that a pair is a sequence
entropy pair if and only if it is $f$-nonseparable.

Moreover, some equivalent conditions of positive entropy which
relate to the number ``$3$" are obtained. It is shown that for an interval map if it is topological null,
then the pattern entropy
of every open cover is of polynomial order, answering a question by
Huang and Ye when the space is the closed unit interval.
\end{abstract}
\maketitle

\section{Introduction}
The study of the complexity or chaotic behavior is a central topic
in topological dynamics. Starting from the work of Li and York
\cite{LY75} various authors introduce a lot of definitions of chaos
according to their understanding of the phenomena. Among them,
Li-Yorke chaos, Denavey chaos \cite{D89} and positive entropy
\cite{Block1992} are popular ones. It is important to understand
their relationships. Recently, it has been shown that for a general
topological dynamical system, Devaney chaos implies Li-Yorke chaos
\cite{HY02} and positive entropy implies Li-Yorke chaos
\cite{BGKM02}.

\medskip

In the study of the so called ``local entropy theory" (for a survey
see \cite{GY09}), a lots of notions are introduced to describe
dynamical properties. It is not clear the relationship of those
properties (related to entropy) with the chaotic behaviors for a
given space. The purpose of the current paper is to study the
relationship in the case when the given space is a closed interval.
We believe that many results of the paper hold for a graph map even
more general spaces.

\medskip

To state our results, we introduce some notations first. Let $I$ be
the closed unit interval $[0,1]$ and $C(I,I)$ denote the class of
continuous maps of $I$ to itself. For $f\in C(I,I)$, let $f^0$ be
the identity, and for $n\in \mathbb N$, let $f^{n+1}=f^n\circ f$,
where $\mathbb{N}$ stands for the set of positive integers.

\medskip

A point $x\in I$ is called a {\it periodic point} of $f$ with {\it
period} $n$ if $f^n(x) = x$, $f^k(x)\neq x$ for $1\leq k<n$.  A
periodic point with period $1$ is called a {\it fixed point}. The
{\em $\omega$-limit set of $x$}, denoted by $\omega_f(x)$, is the set of
limit points of $\{f^i(x)\}_{i=0}^\infty$. A set $W\subset I$ is
called some {\em $\omega$-limit set} for $f$, if there exists an $x\in I$
such that $W=\omega_f(x)$. Denote the collection of all
$\omega$-limit sets for $f$ by $\omega_f$.

\medskip

A point $x\in X$ is called (1) a {\em recurrent point}, if for every neighborhood $U$ of $x$, there exists some
$n>0$, such that  $f^n(x)\in U$; (2) a {\em strongly recurrent point}, if for every neighborhood $U$ of $x$,
there exists some $N>0$, such that if $f^m(x)\in U$ then $f^{m+k}(x)\in U$ for some $k$ with $0<k\leq N$;
(3) a {\em regularly recurrent point}, if for every neighborhood $U$ of $x$, there exists some
$N>0$, such that $f^{kN}(x)\in U$ for all $k>0$.

Denote $Per(f)$, $Rec(f)$, $SR(f)$ and $RR(f)$ by the set of periodic points, recurrent points,
strongly recurrent points and regularly recurrent points,
respectively. It is well known that $$ Per(f)\subset RR(f)\subset
SR(f)\subset Rec(f).$$

The terminology ``chaos'' was first introduced by Li and Yorke
\cite{LY75}  to describe the complex behavior of trajectories. A
pair $\langle x,y\rangle\in I^2$ is called {\em proximal} if
$\liminf_{n\to \infty}|f^n(x)-f^n(y)|=0$ and is called {\em asymptotic} if
$\lim_{n\to \infty}|f^n(x)-f^n(y)|=0$. A {\em scrambled pair} or {\em Li-Yorke pair} is one that
is proximal but not asymptotic. A pair $\langle x,y\rangle$ is called {\em proper} if $x\neq y$.

For $\delta>0$, a pair $\langle x,y\rangle$ is said to be
{\em $\delta$-scrambled}  if
$$\liminf_{n\to \infty} |f^n(x)-f^n(y)|=0 \text{ and } \limsup_{n\to \infty}|f^n(x)-f^n(y)|\geq \delta.$$
A set $C\subset I$ is called {\em scrambled} (resp. {\em $\delta$-scrambled}) if
any proper pair $\langle x,y\rangle \in C^2$  is scrambled (resp.
$\delta$-scrambled). The map $f$ is called {\em Li-Yorke chaotic} (resp.
{\em $\delta$-Li-Yorke chaotic}) if there exists an uncountable scrambled
set (resp. $\delta$-scrambled set).

In \cite{LY75}, Li and Yorke proved that for an interval map period $3$ implies Li-Yorke chaos. In \cite{JS86}, Jankova and Smital generalized this result as follows: if an interval map has positive entropy, then it is Li-Yorke chaotic.

\medskip
The converse of this result is not true: Xiong \cite{X86} and Smital
\cite{S86} constructed  some interval maps with zero entropy which are
Li-Yorke chaotic.

\medskip
In \cite{S86}, Smital also built some useful tools for zero entropy
interval maps: the periodic portion of an $\omega$-limit set and
$f$-nonseparable points. See \cite{Block1992} for another approach
to the periodic portion of an $\omega$-limit set.
\medskip

In this paper, we discuss those various chaotic properties and their
relationships for  interval maps. In section 2, for preparation we
recall some basic definitions and results for a general dynamical system.
In section 3, we review the structure of the $\omega$-limit set and
build a new approach to the periodic portion of an $\omega$-limit
set.

In section 4,  we deal with zero entropy interval maps. First,
we show that  the proximal relation is an equivalence relation for a zero entropy
interval map. Second, some properties of $f$-nonseparable pairs are
obtained. Third, we discus the relationship between Li-Yorke chaos
and $f$-nonseparable pair. Finally, after
reviewing some recent results on the sequence entropy pair, we show
that for a zero entropy interval map a pair is a sequence entropy
pair if and only if it is $f$-nonseparable.

In section 5, we obtain some
equivalent conditions of positive entropy which relate to the number
``$3$" and show that strongly mixing is equivalent to topological K
for interval maps. In section 6, we show that for an interval map if
it is null, then the pattern entropy of every open cover is of
polynomial order, which give an positive answer for a problem in
\cite{HY09} for interval maps.

\section{Preliminaries}
In this section we briefly review some basic definitions and results for a general dynamical system.
By a {\em topological dynamical system} (TDS for short), we mean a pair $(X, T)$,
where $X$ is a  compact metric space with metric $d$ and $T:X \to X$
is a continuous map.

\begin{defn}
Let $(X, T)$ be a TDS. The system $(X,T)$ (or the map $T$) is called
\begin{enumerate}
\item {\em transitive} if for every two nonempty open subsets $U, V$ of $X$, there exists some $n\geq 0$
such that $T^n U \cap V\neq \emptyset$.
\item {\em weakly mixing} if $T\times T$ is transitive on $X\times X$.
\item {\em strongly mixing} if for every two nonempty open subsets $U, V$ of $X$,
there exists some $N\geq 0$ such that $T^n U \cap V\neq \emptyset$ for all $n\geq N$.
\item {\em minimal} if there is no non-trivial subsystem.
\end{enumerate}
\end{defn}

\begin{defn}
Let  $(X, T)$ be a TDS. The system $(X,T)$ (or the map $T$) is called
\begin{enumerate}
\item {\em sensitive on initial conditions} (or just {\em sensitive}) if there exists $\varepsilon >0$
such that for every nonempty open subset $U$ of $X$, there are $x, y\in U$ and $n\geq 0$ such that $d(T^n(x), T^n(y))\geq \varepsilon$.
\item {\em Devaney chaotic}, if it is transitive and the set of periodic points is dense in $X$.
\end{enumerate}
\end{defn}

In 1965, Adler, Konheim and McAndrew introduced topological entropy for a
TDS. Let $\mathcal C_X^o$ be  the set of finite open covers of $X$.
Given two open covers $\mathcal U$ and $\mathcal V$, let
$$\mathcal U \vee \mathcal V =\{U\cap V: U\in \mathcal U, V\in \mathcal V\}.$$
We define $N(\mathcal U)$ as the minimum cardinality of subcovers of
$\mathcal U$.  The topological entropy of $T$ with respect to
$\mathcal U$ is
$$h(T,\mathcal U)=\lim_{n\to\infty}\frac{1}{n}\log N\left(\bigvee_{i=0}^{n-1}T^{-i}\mathcal U\right).$$
The {\em topological entropy} of $T$ is
$$h(T)=\sup_{\mathcal U\in \mathcal C_X^o}h(T,\mathcal U).$$

In 1974, Goodman introduced sequence topological entropy. 
For an $A=\{0\leq t_1<t_2<\cdots \}\subset \mathbb{Z}_+$ and an open
cover $\mathcal U$ of $X$,  the topological sequence entropy of $T$
with respect to $\mathcal U$ and $A$ is
$$h_A(T,\mathcal U)=\lim_{n\to\infty}\frac{1}{n}\log N\left(\bigvee_{i=0}^{n-1}T^{-t_i}\mathcal U\right).$$
The {\em topological sequence entropy} of $(X,T)$ along $A$ is
$$h_A(T)=\sup_{\mathcal U\in \mathcal C_X^o}h_A(T,\mathcal U).$$

Recently, local entropy theory has aroused great interesting, see
\cite{GY09} for a survey on this topic. The notions of
entropy tuple and sequence entropy tuple of length $n$ were defined
in paper \cite{HY06}. Originally, entropy tuple and sequence entropy
tuple are defined using open covers, now we state the equivalence
definition using the notion of independence set (see \cite{HY06, KL07}).
Recall that a TDS $(X, T)$ is {\em tame} if the cardinal
number of its enveloping semigroup is not greater than the cardinal
number of $\mathbb R$ \cite{G06}, and a TDS $(X,T)$ is {\em null} if
$\sup_A h_A(T)=0$.

\begin{defn}
Let $(X, T)$ be a TDS. For a tuple $\tilde{A}=(A_1,\ldots,A_k)$ of subsets of
$X$, a subset $J \subset \mathbb{Z}_+$  is called an {\em independence
set} for  $\tilde{A}$ if for every nonempty finite subset $I
\subset J$, we have
$$\bigcap_{i\in I} T^{-i}A_{s(i)}\neq \emptyset$$
for all $s\in \{1,\ldots, k\}^I$.
\end{defn}

In \cite{KL07}, Kerr and Li defined IE-tuple, IT-tuple and IN-tuple
(standing for entropy, tame and null, respectively) as follows.

\begin{defn}
A tuple $\tilde{x}=\langle x_1, \ldots, x_k\rangle \in X^k$ is called
\begin{enumerate}
\item an {\em IE-tuple} if for every product neighborhood $U_1 \times\cdots\times U_k$ of $\tilde{x}$
the tuple $(U_1,\ldots, U_k)$ has an independence set of positive density.
\item an {\em IT-tuple} if for every product neighborhood $U_1 \times\cdots\times U_k$  of  $\tilde{x}$
the tuple $(U_1,\ldots, U_k)$ has an infinite independence set.
\item an {\em IN-tuple} if for every product neighborhood $U_1 \times\cdots\times U_k$ of $\tilde{x}$
the tuple $(U_1,\ldots, U_k)$  has arbitrarily long finite independence sets.
\end{enumerate}
\end{defn}

Recall that a tuple $\langle x_1, \ldots, x_k\rangle \in X^k$ is
said to be {\em essential} if $x_i\neq x_j$  for all $1\leq i<j\leq k$ and
{\em non-diagonal} if there are $1\leq i<j\leq k$ such that $x_i\neq x_j$.
\begin{defn}
Let $(X, T)$ be a TDS. The system $(X,T)$ is called
\begin{enumerate}
\item {\em uniformly positive entropy} if every essential pair $\langle x_1,  x_2\rangle \in X^2$ is an IE-tuple.
\item {\em topological K} if every essential $k$-tuple $\langle x_1, \ldots, x_k\rangle \in X^k$
is an IE-tuple for all $k\geq 2$.
\end{enumerate}
\end{defn}

\begin{thm} \cite{HY06,KL07} Let $(X, T)$ be a TDS.
\begin{enumerate}
\item A tuple is an entropy tuple iff it is a non-diagonal IE-tuple. In particular
the system $(X, T)$ has zero entropy iff every IE-pair is diagonal.
\item A tuple is a sequence entropy tuple iff it is a non-diagonal IN-tuple.
In particular the system  $(X, T)$ is null iff every IN-pair is diagonal.
\item  The system $(X, T)$ is tame iff every IT-pair is diagonal.
\end{enumerate}
\end{thm}

\section{The structure of $\omega$-limit sets}

We first recall some classical results on the structure of $\omega$-limit sets for interval maps.
The following result is well known, see \cite{Block1992} for example.

\begin{thm}\label{interval-zero-entropy}
Let $f\in C(I,I)$. The following conditions are equivalent:
\begin{enumerate}
\item $h(f)=0$;
\item the period of every periodic point is a power of $2$;
\item every $\omega$-limit set can not properly contain a periodic orbit.
\end{enumerate}
\end{thm}

The following result first appeared in \cite{Sh66}, see also in \cite{BBHS96}.
\begin{thm}\label{interval-sequence-limit-set}
Let $f\in C(I,I)$.
\begin{enumerate}
\item If $\omega_1$ and $\omega_2$ are two $\omega$-limit sets and $a\in \omega_1\cap \omega_2$ is
a limit point from the left (resp., from the right) of both  $\omega_1$ and $\omega_2$,
then $\omega_1\cup \omega_2$ is also an $\omega$-limit set of $f$.
\item If $\omega_1\subset \omega_2\subset \cdots$ is a sequence of $\omega$-limit sets,
then $\overline{\bigcup_{i=1}^\infty \omega_i}$ is also an $\omega$-limit set of $f$.
\end{enumerate}
\end{thm}

On the basis of the above result and Zorn's Lemma we have:
\begin{prop}
Let $f\in C(I,I)$ and $\omega_f$ partially ordered by the inclusion relation.
Then each maximal chain in  $\omega_f$ has a maximal element.
\end{prop}

\begin{lem}\label{interval-maximal-omega-set}
Let $f\in C(I,I)$ with $h(f)=0$. If $\omega_f(x)$ and $\omega_f(x)$
are  two maximal $\omega$-limit  sets, then $\omega_f(x)$ and
$\omega_f(x)$ either coincide or are disjoint.
\end{lem}
\begin{proof}
Assume that $P=\omega_f(x)\cap\omega_f(y)\neq \emptyset$. If $P$ is
finite, then $P$ contains  a periodic point since $f(P)\subset P$.
By Proposition \ref{interval-zero-entropy}(3), both
$\omega_f(x)$ and $\omega_f(y)$ are periodic orbits, which implies
$\omega_f(x)=\omega_f(y)$. If $P$ is infinite, then any limit point
of $P$  is a limit point from the left or from the right of both $\omega_f(x)$ and
$\omega_f(y)$. By Proposition \ref{interval-sequence-limit-set}(1),
$\omega_f(x)\cup \omega_f(y)$ is also an $\omega$-limit set. Then the maximality of $\omega_f(x)$ and
$\omega_f(y)$ imply that they coincide.
\end{proof}

\begin{prop}\cite[Lemma VI.14]{Block1992} \label{interval-periodic-portion-2}
Let $f\in C(I,I)$ with $h(f)=0$ and $x\in I$. Suppose $\omega_f(x)$
is infinite.  For every $k\geq 1$ and $i=0,1,\ldots, 2^k-1$, let
$$J_k^i=[\min \omega_{f^{2^k}}(f^i(x)), \max \omega_{f^{2^k}}(f^i(x))]. $$
Then
\begin{enumerate}
\item $f(J_k^i)\supset J_k^{i+1(\text{mod } 2^k)}$,
\item the closed intervals $(J_k^i)_{0\leq i < 2^k}$ are pairwise disjoint,
\item $J_{k+1}^i\cup J_{k+1}^{2^k+i}\subset J_k^i$ for $0\leq i<2^k$. Both $J_{k+1}^i$
and $J_{k+1}^{2^k+i}$ have an endpoint in common with $J_k^i$,
\item for every $0\leq i<2^k$, $J_{k}^i$ contains a periodic point of period $2^k$,
but no periodic point of period less than $2^k$.
\end{enumerate}
\end{prop}

We call those intervals $(J_k^i)_{k\geq 1, 0\leq i <2^k}$ are a
{\em periodic portion} of  the $\omega$-limit set $\omega_f(x)$. The
periodic portion of an $\omega$-limit set does not depend on the
choose of the base point, i.e. if $\omega_f(x)=\omega_f(y)$, then
they have the same periodic portion.

\begin{thm}\cite[Lemma VI.16, VI.18]{Block1992}\label{interval-structure-of-limit-set}
Let $f\in C(I,I)$ with $h(f)=0$ and $x\in I$. Suppose $\omega_f(x)$
is infinite.  Let $(J_k^i)_{k\geq 1, 0\leq i <2^k}$ be the periodic
portion of $\omega_f(x)$ and
$$C(x)=\bigcap_{k=1}^\infty \bigcup_{i=0}^{2^k-1} J_k^i. $$
Then $C(x)$ is closed, $f(C(x))=C(x)$ and $C(x)\cap Per(f)=\emptyset$.

Moreover, for every nested sequence $J_1^{i_1}\supset
J_2^{i_2}\supset \cdots$,  put $K= \bigcap_{k=1}^\infty J_k^{i_k}$,
then exactly one of the following alternatives holds:
\begin{enumerate}
\item $K=\{y\}\subset\omega_f(x)$ and $y$ is regularly recurrent,
\item $K=[y,z]$, $K\cap\omega_f(x)=\{y,z\}$ and both endpoints of $K$ are strongly recurrent but not regularly recurrent,
\item $K=[y,z]$, $K\cap\omega_f(x)=\{y,z\}$ and one endpoint of $K$ is regularly recurrent and the other is not recurrent.
\end{enumerate}
In every case, if $y\in \omega_f(x)$, then $y$ is regularly recurrent iff $\lim_{n\to\infty}f^{2^n}(y)= y$.
\end{thm}

\begin{thm}\cite[Theorem VI.30]{Block1992}\label{interval-adding-machine}
Let $f\in C(I,I)$ with $h(f)=0$ and $x\in I$. If $Y=\omega_f(x)$ is
infinite, then there  exists a continuous map $\phi$ from $Y$ onto
the adding machine $J$ such that except at most countable points in
$J$ which have two preimages, other points have exact one preimage
and
$$\phi\circ f(y)= \tau \circ \phi(y), \ \forall y\in Y.$$
Moreover, $\phi$ maps $Y$ homeomorphically onto $J$ iff every point $y\in Y$ is regularly recurrent iff $(Y,f)$ is a minimal subsystem.
\end{thm}

\begin{lem}\label{interval-key-lemma}
Let $f\in C(I,I)$ with $h(f)=0$ and $x\in I$. Suppose $\omega_f(x)$
is infinite.  Let $(J_k^i)_{k\geq 1, 0\leq i <2^k}$ be the periodic
portion of $\omega_f(x)$. For every $k\geq 1$, if $J_k^{r_1}$ and $J_k^{r_2}$
are two intervals in $(J_k^i)_{0\leq i <2^k}$,
then  there exists some periodic point between $J_k^{r_1}$ and $J_k^{r_2}$.
\end{lem}
\begin{proof} Recall that for every $k\geq 1$ there exists a periodic point of
periodic $2^k$ in  $J_k^i$ but no periodic point of periodic $2^j$ for any $j<k$ (see Proposition \ref{interval-periodic-portion-2}).

We use induction to show the result.

If $k=1$, let $J_0=[\min \omega_{f}(x), \max \omega_{f}(x)]$. Since
$f(J_0)\supset J_0$,  there is a fixed point $p$ in $J_0$, but
neither $J_1^0$ or $J_1^1$ can contain a fixed point. Then $p$ lies
between $J_1^0$ and $J_1^1$.

Assume for all $k\leq n$, the conclusion is established.

Let $k=n+1$,  we have two cases.

Case 1: $J_{n+1}^{r_1}$ and $J_{n+1}^{r_2}$ are contained in the same
$J_{n}^{r_3}$.  Since $f^{2^n}(J_n^{r_3})\supset J_n^{r_3}$, there
exists a periodic point $p$ in $J_n^{r_3}$ of periodic $2^n$, but
neither $J_{n+1}^{r_1}$ nor $J_{n+1}^{r_2}$ can contain a periodic point
of periodic $2^n$. Then $p$ lies between $J_{n+1}^{r_1}$ and $J_{n+1}^{r_2}$.

Case 2:  $J_{n+1}^{r_1}$ and $J_{n+1}^{r_2}$ are contained in
$J_{n}^{r_3}$ and $J_{n}^{r_4}$ respectively. By induction,
there exists a periodic point $p$ between $J_n^{r_3}$ and
$J_n^{r_4}$. However, $p$ also lies between $J_{n+1}^{r_1}$ and $J_{n+1}^{r_2}$.
\end{proof}
The following Lemma was proved in \cite{Ruette2002}, for completeness we provide a proof.
\begin{lem}\label{interval-Ruette-lemma}
Let $f\in C(I,I)$ with $h(f)=0$ and $x\in I$. Suppose $\omega_f(x)$ is infinite.
\begin{enumerate}
\item If $J$ is an interval containing three distinct points of $\omega_f(x)$,  then $J$ contains a periodic point.
\item If $U$ is an open interval such that $U\cap \omega_f(x)\neq \emptyset$, then there exists $n\geq 0$ such that $f^n(U)$ contains a periodic point.
\end{enumerate}
\end{lem}
\begin{proof}
Let $(J_k^i)_{k\geq 1, 0\leq i <2^k}$ be the periodic portion of $\omega_f(x)$.

(1) Let $J\cap \omega_f(x)\supset \{x_1, x_2, x_3\}$. Without loss
of generality,  assume $x_1<x_2<x_3$. Then there exists a  nested
sequence $J_1^{i_{j_1}}\supset J_2^{i_{j_2}}\supset \cdots$ such
that $x_j\in \bigcap_{k=1}^\infty J_k^{i_{j_k}}$ for $j=1,2,3$. Let
$K_j= \bigcap_{k=1}^\infty J_k^{i_{j_k}}$, then $K_1\cap
K_3=\emptyset$, since $\#(K_j\cap \omega_f(x))\leq 2$ for $j=1,2,3$.
Then by Lemma \ref{interval-key-lemma} there exists a periodic point
$p$ between $K_1$ and $K_3$. Hence, $p\in J$ since $J$ is connected.

(2) Let $z\in U\cap \omega_f(x)$. If $U$ contains a periodic point,
then the proof is complete. Otherwise, let $V \supset U$ be the
maximal subinterval that contains no periodic point. By assumption,
there are $0<n_1<n_2<n_3$ such that $f^{n_j}(x)\in U$ for $j=1,2,3$.
The points $\{z, f^{n_3-n_1}(z), f^{n_2-n_1}(z)\}$ are distinct and
contained in $\omega_f(x)$. If  $\{z, f^{n_3-n_1}(z),
f^{n_2-n_1}(z)\} \subset V$, then the first part of the proof
implies that $V$ contains a periodic point, which contradicts the
definition of $V$. Thus, there exists $j\in \{2,3\}$ such that
$f^{n_j-n_1}(z)\not \in V$. The interval $f^{n_j-n_1}(U)$ contains
both $f^{n_j}(x)\in V$ and $f^{n_j-n_1}(z)\not\in V$. Therefore,
$f^{n_j-n_1}(U)$ must contains a periodic point by the maximality
of the interval $V$.
\end{proof}

\begin{rem}\label{interval-periodic-portion-3}
Let $f\in C(I,I)$ with $h(f)=0$ and $x\in I$. Suppose $\omega_f(x)$
is infinite.  Let $(J_k^i)_{k\geq 1, 0\leq i <2^k}$ be the periodic
portion of $\omega_f(x)$. For every $k\geq 1$, since $(J_k^i)_{0\leq
i< 2^k}$ are pairwise disjoint closed intervals, let
$s_k=\min\{\frac{1}{k}, \frac{1}{4}d(J_k^i, J_k^{i+1}): 0\leq
i<2^k-1\}>0$ and $J_{k,s_k}^i=B(J_k^i, s_k)$, where $d(J_k^i,
J_k^{j})=\inf\{|x-y|:x\in J_k^i, y\in J_k^j\}$ and $B(J_k^i,
s_k)=\{x\in I: |x-y|<s_k $  for some $y\in J_k^i\}$. Then  $(J_{k,
s_k}^i)_{0\leq i< 2^k}$ are pairwise disjoint open intervals and
$d(J_{k,s_k}^i, J_{k,s_k}^{i+1})>s_k$ for $0\leq i<2^k-1$. For
convenience, we also call $(J_{k, s_k}^i)_{k\geq 1, 0\leq i< 2^k}$ is the
{\em periodic portion} of $\omega_f(x)$.
\end{rem}

\begin{lem}\label{interval-omega-lim-lem}
Let $f\in C(I,I)$ with $h(f)=0$ and  $x\in I$. Suppose $\omega_f(x)$
is infinite.  Let $(J_{k,s_k}^i)_{k\geq 1, 0\leq i <2^k}$ be the
periodic portion of $\omega_f(x)$. If $y\in I$ with
$\omega_f(y)\subset \omega_f(x)$, then for every $k\geq 1$ there
exists some $n_k\geq 0$ such that  $f^{n}(y)\in
J_{k,s_k}^{n-n_k(\text{mod }2^k)}$ for all $n\geq n_k$.
\end{lem}
\begin{proof}
Since $\omega_f(y)\subset \omega_f(x)$, then for every $k\geq 1$
there exists some $n_{k_0}\geq 0$ such that $f^n(y)\in
\bigcup_{i=0}^{2^k-1} J_{k,s_k}^i$ for all $n\geq n_{k_0}$.

{\bf Claim}: For every $k\geq 1$ there exists $n_{k_1}\geq n_{k_0}$, such that for every
$n\geq n_{k_1}$,  if $f^n(y)\in J_{k,s_k}^i$, then $f^{n+1}(y)\in
J_{k,s_k}^{i+1(\text{mod }2^k)}$.

{\bf Proof of the Claim}. If not, then there exists a sequence
$\{n_q\}$,  such that $f^{n_q}(y)\in J_{k,s_k}^{i_q}$ and
$f^{n_q+1}(y)\not \in J_{k,s_k}^{i_q+1(\text{mod }2^k)}$.
Without loss of generality, we assume $\lim_{q\to\infty}f^{n_q}(y)=a$,
$f^{n_q}(y)\in J_{k,s_k}^{i_0}$ and $f^{n_q+1}(y)\in
J_{k,s_k}^{i_1}$ with $i_1\neq i_0+1 (\text{mod }2^k)$. Then $a\in
J_k^{i_0}$ but $f(a)\in J_k^{i_1}$. This contradicts the fact that
$f(\omega_f(x)\cap J_k^i)=\omega_f(x)\cap J_k^{i+1(\text{mod
}2^k)}$. Then the proof of the Claim is complete.

Now, choose $n_k\geq n_{k_1}$ such that $f^{n_k}(y)\in J_{k,s_k}^0$,
then  $f^{n}(y)\in J_{k,s_k}^{n-n_k(\text{mod }2^k)}$ for all
$n\geq n_k$.
\end{proof}

\section{Chaos for zero entropy maps}
Throughout this section, if  without any other statements, we assume that $f\in C(I,I)$ with $h(f)=0$.
\subsection{Proximal relation}
First, we consider the proximal relation of $f$. If $\langle
x,y\rangle \in I^2$  is proximal, then $\omega_f(x)\cap
\omega_f(y)\neq \emptyset$.  Two maximal $\omega$-limit sets containing
$\omega_f(x)$ and $\omega_f(y)$ respectively are not disjoint, then
they coincide by Lemma \ref{interval-maximal-omega-set}.

We define a ``kneading sequence'' for one point according to the
periodic portion  of its maximal $\omega$-limit set, which can
characterize the proximal pair following the idea in  \cite{BL95}.
\begin{defn}\label{interval-defn-ckx}
Let $f\in C(I,I)$ with $h(f)=0$ and $x\in I$. If $\omega_f(x)$ is
infinite,  then there exists a unique maximal $\omega$-limit set
$\omega_0$ which contains $\omega_f(x)$. Let $(J_{k, s_k}^i)_{k\geq
1, 0\leq i <2^k}$ be the periodic portion of $\omega_0$. By Lemma
\ref{interval-omega-lim-lem}, for every $k\geq 1$ we can define
$$c_x(k)=\min\left\{n_k\in \mathbb N:\ f^n(x)\in J_{k,s_k}^{n-n_k(\text{mod }2^k)}
\text{ for all } n\geq n_k\right\}$$
\end{defn}

It is easy to see that $c_x(k)=c_x(k+1)\ (\text{mod }2^k)$.

\begin{prop}\label{interval-prox-eq}
Let $f\in C(I,I)$ with $h(f)=0$ and  $ x,y \in I$. If $\omega_f(x)$
and $\omega_f(y)$ are  contained in the same maximal limit set
$\omega_0$ which is infinite, then the following conditions are
equivalent:
\begin{enumerate}
\item $\langle x,y \rangle $ is proximal;
\item $c_x(k)=c_y(k) (\text{mod }2^k)\text{ for all } k\geq 1$.
\end{enumerate}
\end{prop}
\begin{proof}
Since $\omega_0$ is infinite, $\omega_f(x)$ and $\omega_f(y)$ are
also infinite.  Let $(J_{k, s_k}^i)_{k\geq 1, 0\leq i <2^k}$ be the
periodic portion of $\omega_0$.

$(1)\Rightarrow (2)$. If not, there exists some $k\geq 1$ such that
$c_k(x)\neq c_k(y) (\text{mod }2^k)$. Let $\delta=\min_{0\leq
i<2^k-1}d(J_{k,s_k}^i, J_{k,s_k}^{i+1})>0$. Then it is easy to
verify that $|f^n(x)-f^n(y)|\geq \delta$ for all $n\geq
\max\{c_k(x), c_k(y)\}$. This is a contradiction,
since $\langle x,y \rangle$ is proximal.

$(2)\Rightarrow (1)$. For every $\varepsilon>0$, there exists an
appropriate $J_{k,s_k}^i$ such that $\text{diam}(J_{k,s_k}^i)\leq
\varepsilon$. Without loss of generality, we assume $c_k(x)\geq
c_k(y)$. Then $f^{c_k(x)}(x)$, $f^{c_k(x)}(y)\in J_{k,s_k}^0$ and
$f^{c_k(x)+i}(x)$, $f^{c_k(x)+i}(y)\in J_{k,s_k}^i$. Thus,
$|f^{c_k(x)+i}(x)- f^{c_k(x)+i}(y)|\leq \varepsilon$, which implies
$\langle x,y \rangle $ is proximal since $\varepsilon$ is arbitrary.
\end{proof}

Let $A$ be a subset of $\mathbb N$, we call $A$ has {\em Banach density $1$} if
$$\lim_{n\to \infty}\frac{\#(A\cap E_n)}{\#(E_n)}=1$$
for any sequence $\{E_n\}$ of intervals of positive integer,
where $E_n=\{a_n, a_n+1,\ldots, b_n\}$ and $\lim_{n\to\infty}\#(E_n)=\lim_{n\to\infty}(b_n-a_n+1)=\infty$.

Let $\mathcal F_{bd1}=\{A\subset \mathbb N:\ A$  has Banach density
$1\}$.  A pair $\langle x,y\rangle\in I^2$  is called {\em $\mathcal
F_{bd1}$-proximal} if $\{n\in \mathbb N:\ d(f^n(x),
f^n(y))<\frac{1}{p}\}\in \mathcal F_{bd1}$ for all $p\geq 1$.
Clearly, if  $\langle x,y\rangle \in I^2$ is asymptotic, then
$\langle x,y\rangle$ is $\mathcal F_{bd1}$-proximal.
\begin{prop}\label{interval-bd1-proximal}
Let $f\in C(I,I)$  with $h(f)=0$ and $x,y\in I$. If $\langle
x,y\rangle$ is  proximal, then $\langle x,y\rangle $ is  $\mathcal
F_{bd1}$-proximal.
\end{prop}
\begin{proof}
If $\omega_f(x)$ is finite, then it is easy to see that $\langle x,y\rangle$ is
asymptotic. Now assume that $\omega_f(x)$ is infinite, then so is
$\omega_f(y)$. Let $\omega_0$  be the maximal $\omega$-limit set
which contains both $\omega_f(x)$ and $\omega_f(y)$, then $\omega_0$
is also infinite.  Let $(J_{k, s_k}^i)_{k\geq 1, 0\leq i <2^k}$ be
the periodic portion of $\omega_0$.

Fix $p\geq 1$ and $k\geq 1$. Without loss of generality, assume
$c_k(x)\geq c_k(y)$.   Since the intervals $(J_{k,s_k}^i)_{0\leq i<2^k}$
are pairwise disjoint, there are at most $p$ distinct sets
$J_{k,s_k}^i$ with $\text{diam}(J_{k,s_k}^i)\geq 1/p $. Let
$\{E_n\}$ be a sequence of intervals of positive integer and
$\lim_{n\to\infty}\#(E_n)=\infty$. Then
\begin{align*}
\frac{\#\left\{i\in E_n: |f^i(x)- f^i(y)|< \frac{1}{p}\right\}}{\#(E_n)} &\geq
\frac{ \#\left\{i\in E_n: \text{diam}( J_{k,s_k}^{i(\text{mod }2^k)})< \frac{1}{p}\right\}-c_x(k)}{\#(E_n)}\\
& \geq 1-\frac{\left(\frac{\#(E_n)}{2^k}+2\right)p +c_x(k)}{\#(E_n)}\\
& \to 1-\frac{p}{2^k},\qquad \text{as }n\to \infty.
\end{align*}
Thus,
$$\lim_{n\to\infty}\frac{\#\left\{i\in E_n: |f^i(x)-f^i(y)|< \frac{1}{p}\right\}}{\#(E_n)}=1$$
since $k$ is arbitrary. Therefore, $\langle x,y\rangle$ is $\mathcal F_{bd1}$-proximal since $p$ and $\{E_n\}$ are  arbitrary.
\end{proof}

\begin{thm} \label{intevla-proximal-equivalence}
Let $f\in C(I,I)$ with $h(f)=0$. Then the proximal relation $Prox_f=\{\,\langle x, y\rangle \in I^2: \langle x,y \rangle \text{ is proximal}\,\}$ is an equivalence relation.
\end{thm}
\begin{proof}
The reflexivity and symmetry of $Prox_f$ is obviously. The
transitivity of $Prox_f$  is followed by Proposition
\ref{interval-bd1-proximal} and the property that the intersection
of two  Banach density $1$ sets also has Banach density $1$.
\end{proof}

For every $x\in I$, we define the
{\em proximal cell} of $x$ as $Prox_f(x)=\{y\in I: \langle
x,y\rangle$ is proximal$\}$. It is easy to see that $\{Prox_f(x): x\in
I\}$ is a partition of $I$.

\begin{prop}\label{interval-proximal-class}
Let $f\in C(I,I)$ with $h(f)=0$. If $Prox_f(x), Prox_f(y)$ are two
different  proximal cells, then there exists
$\delta>0$, such that
$$\liminf_{n\to \infty} d(f^n(u), f^n(v))\geq \delta$$
for all $u\in Prox_f(x), v\in Prox_f(y)$.
\end{prop}
\begin{proof}
Let $\omega_1$, $\omega_2$ be two maximal $\omega$-limit sets which contain $\omega_f(x)$ and $\omega_f(y)$ respectively.
If $\omega_1\neq \omega_2$, then $\omega_1\cap \omega_2=\emptyset$.
Let $\delta=d(\omega_1, \omega_2)>0$. Then $\liminf_{n\to \infty} d(f^n(u), f^n(v))\geq \delta$ for all $u\in Prox_f(x), v\in Prox_f(y)$.

If $\omega_1=\omega_2$, then there are two cases.

Case 1. $\omega_1$ is a periodic orbit. Then there exist two
periodic points $p_1$ and  $ p_2$, such that
$\lim_{n\to\infty}|f^n(x)-f^n(p_1)|=0$ and
$\lim_{n\to\infty}|f^n(y)-f^n(p_2)|=0$. Since $ \langle x,y\rangle $
is not proximal, one has $p_1\neq p_2$. Let $\delta=\min_{n\geq
0}|f^n(p_1)-f^n(p_2)|>0$. Then $\liminf_{n\to \infty} d(f^n(u),
f^n(v))\geq \delta$ for all $u\in Prox_f(x), v\in Prox_f(y)$.

Case 2. $\omega_1$ is infinite. Let $(J_{k, s_k}^i)_{k\geq 1, 0\leq
i <2^k}$ be the periodic  portion of $\omega_1$. Since $\langle
x,y\rangle$ is not proximal, there exists some $k\geq 1$ such that
$c_x(k) \neq c_y(k)\ (\text{mod }2^k)$. Let
$\delta=\frac{1}{2}\min\{d(J_{k,s_k}^i, J_{k,s_k}^{i+1}):\ 0\leq
i<2^k-1\}>0$. Then $\liminf_{n\to \infty} d(f^n(u), f^n(v))\allowbreak\geq
\delta$ for all $u\in Prox_f(x), v\in Prox_f(y)$.
\end{proof}

\begin{rem} (1). It has been shown that the proximal relation is an equivalence relation for
a zero entropy interval map, but it may not be closed in $I\times I$. For example, let $f(x)=x^2$,
then $\langle 0, x\rangle $ is proximal for all $x\in [0,1)$, but $\langle 0, 1\rangle$ is not proximal.

(2). Recall that a pair $\langle x, y\rangle\in I^2$ is called {\em regionally proximal}, if for every
$\varepsilon>0$, there exist $u, v$ and $n\geq 1$ such that $|x-u|<\varepsilon, |y-v|<\varepsilon$
and $|f^n(u)-f^n(v)|<\varepsilon$. Obviously, every proximal pair is regionally proximal.
Let $Q_f$ denote the set of regionally proximal pairs. It is not hard to verify that
$$Q_f=\bigcap_{k=1}^{\infty}\overline{\bigcup_{n=k}^{\infty}(f\times f)^{-n}V_k}, $$
where $V_k=\{\langle x,y\rangle \in I^2:\ |x-y|<1/k\}$. Then $Q_f$ is a closed subset
of $I^2$, but it may not be an equivalence relation. For example, let $f(x)=2x^2$ on $[0,1/2]$
and $2(x-1/2)^2+1/2$ on $[1/2,1]$, then $\langle 0, 1/2 \rangle$ and $\langle 1/2, 1\rangle $
are regionally proximal, but $\langle 0, 1\rangle$ is not regionally proximal.
\end{rem}

\subsection{$f$-nonseparable pair}
\begin{defn}
Let $f\in C(I,I)$ with $h(f)=0$. A proper pair $\langle u,v\rangle \in I^2$ is called
{\em $f$-nonseparable}, if there exists some $x\in I$ with infinite $\omega_f(x)$ and let
$(J_k^i)_{k\geq 1, 0\leq i <2^k}$ be the periodic portion of $\omega_f(x)$, such that
there exists a nested sequence $J_1^{i_1}\supset J_2^{i_2}\supset \cdots$ satisfying
$\bigcap_{k=1}^\infty J_k^{i_k}=[u,v]\text{ or }[v,u]$.
\end{defn}

The definition of $f$-nonseparability first appeared in \cite{S86}. Our definition is
slightly different from the origin one, but it is easy to see that they are equivalent.

\begin{lem}
Let $f\in C(I,I)$ with $h(f)=0$ and $u<v\in I$. If $\langle u,v\rangle$ is $f$-nonseparable, then
\begin{enumerate}
  \item $[u,v]$ is wandering, i.e. $f^n([u,v])\cap [u,v]=\emptyset$ for all $n\geq 1$,
  \item $u,v$ are not the endpoints of $I$, i.e. \@$u,v\in (0,1)$.
\end{enumerate}
\end{lem}
\begin{proof}
(1) By the structure of $\omega$-limit set (see Theorem \ref{interval-structure-of-limit-set}),
it is easy to see that  $[u,v]$ is wandering.

(2) Let $u,v \in \omega_f(x)$. Then $f^n([u,v])\cap Per(f)=\emptyset$ for all $n\geq 0$ since
$[u,v]\in C(x)$. If $u=0$, then $[u,v)$ is open in $I$. By Lemma \ref{interval-Ruette-lemma}
there exists some $m\geq 0$ such that  $f^m([u,v))\cap Per(f)\neq \emptyset$. This is a contradiction.
Thus, $u>0$ and one can show $v<1$ similarly.
\end{proof}

A point $x\in I$ is called an {\em eventually periodic point},
if there exists some $n\geq 0$ such that $f^n(x)$ is
a periodic point. Denote $EP(f)$ by the set of eventually periodic points. Then
$\overline{EP(f)}\supset \bigcup_{x\in I}\omega_f(x)$. Moreover,
if $u<v\in I$ and  $\langle u,v\rangle$ is $f$-nonseparable, then
$[u,v]\cap EP(f)=\emptyset$. But $\{u,v\}\subset
\overline{EP(f)}$, i.e. $u$ is a limit point of $EP(f)$ from the
left, $v$ is a limit point of $EP(f)$ from the right. Therefore,
for every $x\in I$ there is at most one point such that they can
form a $f$-nonseparable pair.

\begin{prop}
Let $f\in C(I,I)$ with $h(f)=0$ and $x\in I$. Suppose that $\omega_f(x)$ is infinite and  $u, v \in \omega_f(x)$.
\begin{enumerate}
\item If $\langle u,v\rangle$ is $f$-nonseparable, then so is $\langle f(u), f(v)\rangle $.
\item There exists a unique pair $\langle y,z \rangle \in \omega_f(x)\times \omega_f(x)$
such that $f(y)=u, f(z)=v$. Moreover, if $\langle u,v\rangle$ is $f$-nonseparable, then so is $\langle y,z\rangle$.
\end{enumerate}
\end{prop}
\begin{proof}
By Theorem \ref{interval-structure-of-limit-set}, Theorem \ref{interval-adding-machine} and the definition of $f$-nonseparability.
\end{proof}

\begin{prop}\label{interval-f-nonsep-eq}
Let $f\in C(I,I)$ with $h(f)=0$ and $u<v\in I$. Then the following conditions are equivalent:
\begin{enumerate}
\item $\langle u,v\rangle$ is $f$-nonseparable;
\item there exists an $\omega$-limit set $\omega_0$ which is infinite such that $\{u,v\}\in \omega_0$ and $(u,v)\cap Per(f)=\emptyset$.
\end{enumerate}
\end{prop}
\begin{proof}  By the definition of $f$-nonseparability and Theorem \ref{interval-structure-of-limit-set},
(1)$\Rightarrow$(2) is obvious. It remains to show
(2)$\Rightarrow$(1). Let $(J_k^i)_{k\geq 1, 0\leq i <2^k}$ be the
periodic portion of $\omega_0$. If $\langle u, v\rangle$ is not
$f$-nonseparable, then there exist some $k\geq 1$ and two different
intervals $J_k^{i_1}$ and $J_k^{i_2}$ such that $u\in J_k^{i_1},
v\in J_k^{i_2}$. Then by Lemma \ref{interval-key-lemma} there exists
a periodic point $p$ between $J_k^{i_1}$ and $J_k^{i_2}$. Therefore,
$p\in (u,v)$. This is a contradiction.
\end{proof}

\begin{prop}
Let $f\in C(I,I)$ with $h(f)=0$. Then the set of $f$-nonseparable pairs is either empty or countable.
\end{prop}
\begin{proof}
Suppose that the set of $f$-nonseparable pairs is not empty. Let
$A=\{x\in I: \omega_f(x)$ is a maximal $\omega$-limit set and there
exist $u,v\in \omega_f(x)$  such that $\langle u,v\rangle$ is
$f$-nonseparable$\}$. Then $A$ is at most countable since for every
$x\in A$, $C(x)$ contains a non-degenerate interval and any two
different $C(x)$ and $C(y)$ are pairwise disjoint.
Then by Theorem \ref{interval-structure-of-limit-set}, Theorem
\ref{interval-adding-machine} and Proposition \ref{interval-f-nonsep-eq},
for every $x\in A$, $\omega_f(x)\times \omega_f(x)$ contains exactly countable
$f$-nonseparable pairs .
\end{proof}

\subsection{Chaos in the sense of Li-Yorke}

\begin{lem}\label{intervl-nonseparable-lemma-new}
Let $f\in C(I,I)$ with $h(f)=0$ and $u<v\in I$. Suppose $\langle
u,v \rangle$ is $f$-nonseparable  with respect to $\omega_f(x)$,
$(J_k^i)_{k\geq 1, 0\leq i <2^k}$ and $\bigcap_{k=1}^\infty
J_k^{i_k}=[u,v]$. If $A_1$ and $A_2$ are neighborhoods of $u$ and
$v$ respectively, then there exists some $n\geq 1$ such that $[u,v]$ is
in the interior of $f^{2^n}(A_1)\cap f^{2^n}(A_2)$.
\end{lem}

\begin{proof}
By Lemma \ref{interval-Ruette-lemma}, there exists some $n_j\geq 0$ such
that $f^{n_j}(A_j)$  contains periodic points $y_j$  for $j=1,2$. The
periods of $y_1, y_2$ are some powers of 2. Let $2^p$ be a common
multiple of their periods and let $q>p$ such that $2^q>\max\{n_1,
n_2\}$. Let $z_j=f^{2^{q-n_j}}(y_j)$ for $j=1,2$. Then
$f^{2^p}(z_j)=z_j\in f^{2^{q-n_j}}(f^{n_j}(A_j))=f^{2^q}(A_j)$ for
$j=1,2$. Moreover, $z_j\not \in J_q^{i_q}$ since the period of $z_j$
is less than $2^q$. Suppose for instance that $z_1$ is in the left
of $J_q^{i_q}$, the case in the right being symmetric.

Let $g=f^{2^q}$. There exists $i_{q_0}\geq 0$ such that
$J_{q+1}^{i_{q+1}}\cup J_{q+1}^{i_{q_0}} \subset J_q^{i_q}$,
$g(J_{q+1}^{i_{q+1}})\supset J_{q+1}^{i_{q_0}}$ and
$g(J_{q+1}^{i_{q_0}})\supset J_{q+1}^{i_{q+1}}$,  and a fixed point
$c$ for $g$ between $J_{q+1}^{i_{q+1}}$ and $J_{q+1}^{i_{q_0}}$.

Case 1. $J_{q+1}^{i_{q+1}}$ is in the left of
$J_{q+1}^{i_{q_0}}$, then by  connectedness $g(A_1)\supset
J_{q+1}^{i_{q+1}}$ since $g(u)\in J_{q+1}^{i_{q_0}}$ and $z_1\in
g(A_1)$. Hence, $g^2(A_1)\supset J_{q+1}^{i_{q_0}}\cup \{z_1, c\}$
and by the connectedness $g^2(A_1)\supset [z_1, c]\supset
J_{q+1}^{i_{q+1}}$.

Case 2. $J_{q+1}^{i_{q+1}}$ is in the right of
$J_{q+1}^{i_{q_0}}$, then by the connectedness $g^2(A_1)\supset
J_{q+1}^{i_{q_0}}\cup \{z_1, c\}$ since $z_1\in g^2(A_1)$ and
$g^2(u)\in J_{q+1}^{i_{q+1}}$. Then $g^3(A_1)\supset
J_{q+1}^{i_{q+1}}\cup\{z_1,c\}$ and by the connectedness $g^3(A_1)\supset
[z_1, c]$. Therefore, $[u,v]$ are in the interior of $g^4(A_1)$,
since there exists a point $h\in J_{q+1}^{i_{q_0}}\cap \omega_f(x)$
such that $g(h)=v$ and a neighborhood $D$ of $h$ with $g(D)\subset
A_2$, then $D$ contains an eventually periodic point $e$ and $g(e)$
must be in the right of $v$.

As a conclusion, we get $[u,v]$ is in the interior of $g^4(A_1)$.
Similarly,  we have $[u,v]$ is in the interior of $g^4(A_2)$. Let
$n=q+2$, then $[u,v]$ is in the interior of $f^{2^n}(A_1)\cap
f^{2^n}(A_2)$.
\end{proof}

By Lemma \ref{intervl-nonseparable-lemma-new} and induction, we have the following result.
\begin{prop}\label{interval-covering-map}
Let $f\in C(I,I)$ with $h(f)=0$ and $u,v\in I$. If $\langle u,v
\rangle$ is  $f$-nonseparable,  then there exist two sequences of
closed intervals $\{U_n\}$ and $\{V_n\}$  which are neighborhoods of
$u$ and $v$ respectively, and a sequence of positive number
$\{k_n\}$ such that
\begin{enumerate}
\item  $U_n\supset U_{n+1}$, $\lim_{n\to\infty}\text{diam}(U_n)= 0$,
\item  $V_n\supset V_{n+1}$, $\lim_{n\to\infty}\text{diam}(V_n)=0$,
\item $f^{2^{k_n}}(U_n)\cap f^{2^{k_n}}(V_n) \supset U_{n+1}\cup V_{n+1}$.
\end{enumerate}
\end{prop}

\begin{thm}\label{interval-LY-chaos}
Let $f\in C(I,I)$ with $h(f)=0$, then the following conditions are equivalent:
\begin{enumerate}
\item $f$ is Li-Yorke chaotic;
\item there exists a scrambled pair;
\item there exists a $\delta$-scrambled Cantor set for some $\delta>0$;
\item there exists an $f$-nonseparable pair.
\end{enumerate}
\end{thm}
\begin{proof}
(3)$\Rightarrow$(1)$\Rightarrow$(2) is trivial.

(2)$\Rightarrow$(4). Let $\langle x, y\rangle \in I^2$ be a
scrambled pair. Then $\omega_f(x)$ and  $\omega_f(y)$ are contained
in the same maximal $\omega$-limit set $\omega_0$ which is infinite.
Let $(J_{k,s_k}^i)_{k\geq 1, 0\leq i <2^k}$ be the periodic portion
of $\omega_0$. Since $\limsup_{n\to\infty}|f^n(x)-f^n(y)|>0$, there
exists a sequence $\{n_i\}$ such that
$\lim_{i\to\infty}f^{n_i}(x)=a$ and $\lim_{i\to\infty}f^{n_i}(y)=b$
for $a\neq b$. Then $\langle a,b \rangle$ is $f$-nonseparable. In
fact, if not, then there exist $k\geq 1$ and $i_0\neq i_1$ such that
$a\in J_k^{i_0}$ and  $b\in J_k^{i_1}$. Thus, there are infinitely
many $n_i$ such that $f^{n_i}(x)\in J_{k,s_k}^{i_0}$ and
$f^{n_i}(y)\in J_{k,s_k}^{i_1}$. This is a contradiction since
$\langle x, y\rangle$ is proximal.

(4)$\Rightarrow$(3). Let $\langle u,v\rangle\in I^2$ be an $f$-nonseparable pair and $\delta=|v-u|$.
Let $\{U_n\}$, $\{V_n\}$ and $\{k_n\}$ as in Proposition \ref{interval-covering-map}. Let $t_0=0$ and $t_m=\sum_{n=1}^m 2^{k_n}$ for $m\geq 1$.

First we build a family of closed subintervals $\{E_{a_0a_1,\ldots,a_m}: m\geq 0, a_i\in\{0,1\}\}$ satisfying
 the following properties:
\begin{enumerate}
\item[(a)] $E_{a_0a_1,\ldots,a_ma_{m+1}}\subset E_{a_0a_1,\ldots,a_m}$,
\item[(b)] $E_{a_0a_1,\ldots,a_m}\cap E_{b_0b_1,\ldots,b_m}=\emptyset$ if $a_0a_1,\ldots,a_m\neq b_0b_1,\ldots,b_m$,
\item[(c)] for $\alpha=a_0a_1,\ldots,a_m$, $f^{t_i}(E_\alpha) \subset W_i$ for $i=0,1,\ldots,m-1$
and $f^{t_m}(E_\alpha)=W_m$ where $W_i= U_i$ if $a_i=0$ and $W_i= V_{i}$ if $a_i=1$.
\end{enumerate}

Let $E_0=U_1$, $E_1=V_1$. Suppose that $E_{a_0a_1,\ldots,a_m},
a_i\in\{0,1\}$ are already defined. For $a_0a_2,\ldots, a_ma_{m+1}$,
we have $E_{a_0a_1,\ldots,a_m} \stackrel{f^{t_m}}{\longrightarrow}
W_m \stackrel{f^{2^{k_{m+1}}}}{\longrightarrow} W_{m+1}$, where notation
$A\stackrel{g}{\longrightarrow}B$ means $g(A) \supset B$. Let $F$ be
a subinterval of $W_m$ of minimal length such that
$f^{2^{k_{m+1}}}(F)=W_{m+1}$ and $E_{a_0a_1,\ldots,a_ma_{m+1}}$ be a
subinterval of $E_{a_0a_1,\ldots,a_m}$ of minimal length such that
$f^{t_m}(E_{a_0a_1,\ldots,a_ma_{m+1}})=F$.
Then it is easy to verify that $E_{a_0a_1,\ldots,a_ma_{m+1}}$ satisfies the requirement.

With $\mathbb Z_+=\{0,1,2,\ldots\}$ let $\Sigma$ be the Cantor space
$\{0,1\}^{\mathbb Z_+}$  regard as the set of infinite words. For
every $\alpha=(a_0 a_1, a_2, \ldots)\in \Sigma$, let
$E_\alpha=\bigcap_{m=0}^{\infty}E_{a_0a_1\ldots a_m}$. Then
$E_\alpha$ is either a nonempty compact interval or a single point.
Moreover $E_\alpha\cap E_\beta=\emptyset$ if $\alpha\neq \beta\in \Sigma$.

Let $\Lambda=\{\alpha\in\Sigma: E_\alpha$ is not reduced to a single
point$\}$. The set  $\Lambda$ is at most countable because the sets
$(E_\alpha)_{\alpha\in\Lambda}$ are pairwise disjoint and nondegenerate
intervals. Let
$$X=\bigcup_{\alpha\in \Sigma} E_\alpha\setminus \bigcup_{\beta\in \Lambda} int(E_\beta).$$
It is easy to see that $X$ is a totally disconnected compact set.
Define $\phi:X \to \Sigma$  by $\phi(x)=\alpha$ if $x\in E_\alpha$.
Clearly, the map $\phi$ is well defined, continuous and onto.

Fix $\gamma=(c_0c_1\ldots)=(010101\ldots)$. Define $\psi: \Sigma\to\Sigma$ by
\[\psi((a_n)_{n\in\mathbb Z_+})=(a_0c_0a_0a_1c_0c_1\ldots a_0a_1\ldots a_n c_0c_1\ldots c_n\ldots).\]
The map $\psi$ is clearly continuous, thus $\psi(X)$ is
compact.

For every $\alpha \in \Sigma$, choose $x_\alpha\in X$ such that
$\phi(x_\alpha)=\psi(\alpha)$  and let $S=\{x_\alpha\in X:
\alpha\in\Sigma\}$. If $\psi(\alpha)\not \in \Lambda$ then there is
a unique choice for $x_\alpha$ and if  $\psi(\alpha) \in \Lambda$
then there are two possible choices. Consequently, $S$ is equal to
$\phi^{-1}(\psi(X))$ deprived of a countable set. Let $\alpha,
\beta$ be two distinct elements of $\Sigma$. By the definition of
$\psi$, for every $N\geq 0$ there exists some $m\geq N$ such that the
$m$-th coordinates of $\psi(\alpha)$ and $\psi(\beta)$ are distinct.
Then either $f^{t_m}(x_\alpha)\in U_{\psi(\alpha)_m}$ and $f^{t_m}(x_\beta)\in
V_{\psi(\alpha)_m}$, or the converse. Thus,
\[\limsup_{n\to\infty} |f^n(x_\alpha)-f^n(x_\beta)|\geq \delta.\]

According to the choice of $\gamma$ and the definition of $\psi$,
for every $N\geq 0$  there exists some $m\geq N$ such that for every
$\alpha\in\Sigma$ the $m$-th coordinate of $\psi(\alpha)$ is 0,
which implies $f^{t_m}(x_\alpha)\in U_{\psi(\alpha)_m}$. Thus,
$$\liminf_{n\to\infty}|f^n(x)-f^n(y)|=0$$
for all $x,y\in S$ since $\lim_{n\to\infty}diam(U_n)=\lim_{n\to\infty}diam(V_n)=0$.
By the choice of $\gamma$, it is also easy to see that $\{u,v\}\in\omega_f(x)$ for all $x\in S$.

Since $X$ is totally disconnected and $\phi^{-1}(\psi(X))$ is an
uncountable closed subset  of $X$, then $S$ contains some Cantor subset
$K$. Therefore, $K$ is a $\delta$-scrambled Cantor set.
\end{proof}

\begin{rem} It should be noticed that in the above Theorem {\rm (3)$\Leftrightarrow$(4)}
was proved in \cite{S86} and {\rm (2)$\Leftrightarrow$(3)} was proved in \cite{KS89}, but here we give a new proof.
\end{rem}


\begin{cor}\label{interval-f-non-sep-it}
Let $f\in C(I,I)$ with $h(f)=0$. If $\langle u,v \rangle$ is $f$-nonseparable, then $\langle u, v\rangle$ is an IT-pair.
\end{cor}

\begin{proof}
Let $\{U_n\}$, $\{V_n\}$ and $\{k_n\}$ as in Proposition
\ref{interval-covering-map}.  Let $t_n=\sum_{j=1}^n k_n$. If $U, V$
are neighborhood of $u, v$ respectively, then there exists some $N\geq 1$
such that $U_n\subset U$ and $V_n\subset V$ for all $n\geq N$.

For every $s\in \{0,1\}^{\{0,1,\ldots, k\}}$, by the proof of
Theorem \ref{interval-LY-chaos}  there exists a $w_s$ such that for
$0\leq i \leq k$, $f^{t_{N+i}}(w_s)\in U_{N+i+1}$ if $s(i)=0$ and
$f^{t_{N+i}}(w_s)\in V_{N+i+1}$ if $s(i)=1$, then
$$w_s\in \bigcap_{i=0}^{k}f^{-t_{N+i}}A_{s(i)}$$
where $A_0=U$, $A_1=V$.

Thus, $\{t_N, t_{N+1}, \ldots\}$ is an infinite independence set of $(U, V)$.
Therefore,  $\langle u, v\rangle$ is an IT-pair and $(X,T)$ is not
tame since $\langle u, v\rangle$ is not in the diagonal.
\end{proof}

\begin{defn}
Let $(X,T)$ be a TDS and $S, R\subset X$.
\begin{enumerate}
\item $S$ and $R$ are called {\em equivalent} if there exists a bijection $\phi: S\to R$ such that
$$\lim_{n\to \infty}d(f^n(\phi(x)), f^n(x))=0\quad \text{for any }x\in S.$$
\item $S$ and $R$ are called {\em separable} if there exists $\delta>0$ such that
$$\liminf_{n\to \infty}d(f^n(x), f^n(y))\geq \delta\quad \text{for any } x\in S \text{ and }y\in R.$$
\end{enumerate}
\end{defn}

\begin{prop}\label{interval-scr-eq-sep}
Let $f\in C(I,I)$ with $h(f)=0$. Then every two maximal scrambled sets are either equivalent or separable.
\end{prop}
\begin{proof}
Let $S$ and $R$ be two maximal scrambled sets. If $S$ and $R$ are
contained in different proximal cells, then by
Proposition \ref{interval-proximal-class} there exists $\delta>0$
such that
$$\liminf_{n\to \infty}|f^n(x)-f^n(y)|\geq \delta\quad \text{for any } x\in S \text{ and }y\in R.$$
Now assume that $S$ and $R$ are contained in the same proximal
cell $Prox_f(z)$.  We define a relationship $\sim$ in
$Prox_f(z)$. Let $a,b\in Prox_f(z)$, $a \sim b $ iff $\langle a,b
\rangle$ is asymptotic. Then $\sim$ is a equivalence relation on
$Prox_f(z)$. It is easy to see that in $Prox_f(z)$ every maximal
scrambled set contains exactly one representative point for every
one of those $\sim$ equivalent classes. Thus, there exists a
bijection $\phi: S\to R$ such that
\[\lim_{n\to \infty}|f^n(x)- f^n(\phi(x))|=0\qquad \text{for any }x\in S. \qedhere\]
\end{proof}

\begin{rem}
The Proposition \ref{interval-scr-eq-sep} was proved by Balibrea and Lopez in \cite{BL95}.
It seems that they used some lemmas which are only available for piecewise monotone maps.
\end{rem}
{\bf Quesition}: If $h(f)=0$, is every maximal scrambled set uncountable?

\begin{defn}\cite{X05} Let $(X,T)$ be a TDS and $n\geq 2$.  A tuple $\langle x_1, x_2,\ldots,
x_n\rangle \in X^n$ is called {\em $n$-scrambled} if
$$ \liminf_{k\to\infty}\max_{1\leq i<j\leq n}d\left(f^k(x_i), f^k(x_j)\right)
=0,\quad\limsup_{k\to\infty}\min_{1\leq i<j\leq n}d\left(f^k(x_i), f^k(x_j)\right)>0.$$

The system $(X,T)$ (or the map T) is called {\em $n$-chaotic in the sense
of Li-Yorke} if there exists an uncountable subset $S\subset X$ such
that every essential tuple $\langle x_1, x_2,\ldots,\allowbreak x_n \rangle\in S^n$ is $n$-scrambled.
\end{defn}

\begin{thm}\label{interval-3-srca-tuple}
Let $f\in C(I,I)$ with $h(f)=0$. Then there is no $3$-scrambled tuple.
\end{thm}
\begin{proof}
If $\langle x_1, x_2, x_3\rangle\in I^3 $ is a $3$-scrambled tuple,
then  there exists a sequence $\{n_q\}$ in $\mathbb N$, such that $\lim_{q\to\infty}f^{n_q}(x_i)= a_i$
for $i=1,2,3$ and $a_1$,$ a_2$ and $a_3$ are pairwise distinct. Let
$\omega_0$ be the maximal $\omega$-limit set contains
$\omega_f(x_i)$ for $i=1,2,3$, then $\omega_0$ is infinite. Let
$(J_{k,s_k}^i)_{k\geq 1, 0\leq i <2^k}$ be the periodic portion of
$\omega_0$.  Since $\langle x_1, x_2\rangle $ and $\langle
x_1, x_3\rangle$ are proximal, by Proposition \ref{interval-prox-eq},
 for every $k\geq 1$,
$a_1$, $a_2$ and $a_3$ must be in the same $J_k^{r_k}$. Thus,
$\{a_1,a_2,a_3\}\subset \bigcap_{k=1}^{\infty}J_k^{r_k}$ , which
implies $\#(\bigcap_{k=1}^{\infty}J_k^{r_k} \cap \omega_0)\geq 3$.
This contradicts the  structure of $\omega$-limit set (see
Theorem \ref{interval-structure-of-limit-set}).
\end{proof}

\begin{cor}
There exists a TDS which is $2$-chaotic in the sense of Li-Yorke but
does not  have any $3$-scrambled tuple.
\end{cor}
\begin{proof}
Any interval map, which is chaotic in the sense of Li-Yorke and has
zero entropy,  satisfies the requirement.
\end{proof}

\subsection{Sequence entropy pair}
In \cite{FS91}, Franzova and Smital showed that positive sequence topological entropy can characterize Li-Yorke chaos:
\begin{thm}\label{interval-null-vs-chaos} 
Let $f\in C(I,I)$. Then $f$ is Li-Yorke chaotic iff it is not null.
\end{thm}

\begin{thm}\label{interval-null-eq-tame}\cite{GY09}
Let $f\in C(I,I)$. Then $f$ is null iff it is tame.
\end{thm}

The structure of the set of IN-pairs (or sequence entropy pairs) was studied by Tan, Ye and Zhang in \cite{TYZ09}.

\begin{thm}\label{interval-3-seq-ent-pair}\cite{TYZ09}
Let $f\in C(I,I)$ with $h(f)=0$. If $f$ is not null, then there exist exactly countable IN-pairs, but no essential $3$-IN-tuple.
\end{thm}

\begin{thm}\label{interval-seq-ent-pair-str}\cite{TYZ09}
Let $f\in C(I,I)$ with $h(f)=0$. If $\langle x, y \rangle$ is an IN-pair and $x<y$, then
\begin{enumerate}
\item both $\omega_f(x)$ and $\omega_f(y)$ are infinite,
\item $[x,y]$ is wandering, i.e. $f^n([x,y])\cap [x,y]=\emptyset$ for all $n\geq 1$,
\item $\langle x, y\rangle $ is asymptotic,
\item $[x, y]\cap EPer(f)=\emptyset$, but $\{x, y\}\cap \overline{Per(f)}\neq \emptyset $ and $\{x,y\}\subset \overline{EPer(f)}$.
\end{enumerate}
\end{thm}

\begin{thm}\label{interval-IT-IN-eq}
Let $f\in C(I,I)$ with $h(f)=0$ and $x<y \in I$. Then the following conditions are equivalent:
\begin{enumerate}
\item $\langle x , y \rangle$ is $f$-nonseparable;
\item $\langle x , y \rangle$ is an IT-pair;
\item $\langle x , y \rangle$ is an IN-pair.
\end{enumerate}
\end{thm}

\begin{proof}
(1)$\Rightarrow$(2) is proved in Corollary \ref{interval-f-non-sep-it} and (2)$\Rightarrow$(3) is trivial.

It remains to show (3)$\Rightarrow$(1). Assume that $\langle x , y
\rangle$ is an IN-pair.  Without loss of generality, assume that $x$ is a
limit point of $Per(f)$ from the left. Let $U_1$ and $U_2$ be two
disjoint connected neighborhoods of $x$ and $y$ respectively. Then
there are periodic points $p$, $q$ and $n\geq 1$ such that $p\in
U_1$ and $q\in f^n(U_2)$. Without loss of generality, assume that $p$ and
$q$ are fixed points and $n=1$, since the periods of $p, q$ are the
powers of $2$ and $\langle x,y \rangle$ is also an IN-pair for $f^n$
for every $n\geq 1$.

{\bf Claim: } There exists $n\geq 1$ such that the subinterval $[x,y]$ is in the interior of $f^{n}(U_1)\cap f^{n}(U_2)$.

{\bf Proof of the Claim: }Clearly, we have $p<x$, but there are two cases about the position of $q$.

Case 1, $q>y$. (a) $f(x)>y, f(y)>y$. Since $q$ is a fixed point
and $f(x)>y$, then by the connectedness $[x,y]$ is in the interior of
$f^k(U_1)$ for all $k>1$. Since $\langle x,y\rangle$ is an IN-pair,
there exist $z\in U_2$ and $1<n_1<n_2$ such that $f^{n_1}(z)\in U_1$
and $f^{n_2}(z)\in U_1$. Then  one of $f^{n_1}(z) , f^{n_2}(z)$ must
be on the left side of $x$ since $[x,y]$ is wandering. Since $q>y$, $q$
is a fixed point and $q\in f(U_2)$, by the connectedness we have $[x,y]$
is in the interior of $f^{n_1}(U_2)$ or $f^{n_2}(U_2)$. (b)
$f(x)<x, f(y)<x$. This is the symmetric case of (a).

Case 2, $q<x$. (a) There exists some $k\geq 1$ such that $f^k(x)>y,
f^k(y)>y$. Without loss of generality, we can assume $k=1$. Since
$q$ is a fixed point and $f(x)>y$, then by the connectedness $[x,y]$ is
in the interior of $f(U_1)$. Since $p\in f(U_2)$ and $f(y)>y$, then
by the connectedness $[x,y]$ is in the interior of $f(U_2)$. (b) We
have $f^k(x)<x, f^k(y)<x$ for all $k\geq 1$. Then $f^k([x,y])$ is in
the left of $[x,y]$  for all $k\geq 1$. Thus if $z\in U_2$ and
$f^n(z)\in U_2$ for some $n\geq 1$, then $z>y$. Since $\langle
x,y\rangle$ is an IN-pair, there exist $u\in U_1, v\in U_2$ and
$1<n_1<n_2$ such that $f^{n_1}(u)\in U_2$,  $f^{n_2}(u)\in U_2$ and
$f^{n_1}(v)\in U_2$,  $f^{n_2}(v)\in U_2$. Then $f^{n_1}(u)>y$ and
$f^{n_1}(v)>y$. Thus $[x,y]$ is in the interior of $f^{n_1}(U_1)$
and $f^{n_1}(U_2)$. This completes the proof of the Claim.

Similarly to the usage of Lemma \ref{intervl-nonseparable-lemma-new}
to prove Theorem  \ref{interval-LY-chaos}, we can get that there
exists some $z\in I$ such that $x, y\in \omega_f(z)$. Thus, by
Proposition \ref{interval-f-nonsep-eq} $\langle x, y\rangle$ is
$f$-nonseparable.
\end{proof}

\section{Positive entropy Maps}

\begin{thm} \label{interval-postive-entropy-eq1}
Let $f\in C(I,I)$. Then the following conditions are equivalent:
\begin{enumerate}
\item $h(f)>0$;
\item there exists a subsystem which is Devaney chaotic \cite{L93};
\item there exists $n\geq 1$ such that $f^n$ has a strongly mixing subsystem \cite{XY92};
\item $f$ is distributionally chaotic \cite{SS94}.
\end{enumerate}
\end{thm}

\begin{defn}
\cite{X05} Let $(X,T)$ be a TDS and $n\geq 2$. $(X,T)$ is said to be
{\em $n$-sensitive},  if there exists some $\delta>0$  such that for every
nonempty open subset $U\subset X$, there exist $n$ distinct points
$x_1, x_2, \ldots, x_n \in U$ and $k\geq 1$ satisfying $\min_{1\leq
i< j\leq n}\{d(f^k(x_i), f^k(x_j))\}\geq \delta$.
\end{defn}

\begin{thm} \label{Devaney-chaos-thm}
Let $(X,T)$ be a TDS. If $(X,T)$ is Devaney chaotic, then
\begin{enumerate}
\item $(X,T)$ is infinite sensitive (i.e. $n$-sensitive for all $n\geq 2$) \cite{YZ08},
\item $(X,T)$ is infinite chaotic in the sense of Li-Yorke (i.e. there is an uncountable subset $S$ of X
which is $n$-scrambled for all $n\geq 2$) \cite{X05}.
\end{enumerate}
\end{thm}

\begin{lem} \label{tran-sys-lemma}
Let $(X,T)$ be a TDS. If $(X,T)$ is transitive and has two periodic points, then
\begin{enumerate}
\item $Prox_T$ is not an equivalence relation,
\item there are two maximal scrambled sets which are neither equivalent nor separable.
\end{enumerate}
\end{lem}
\begin{proof}
(1) Without loss of generality, we assume that there are two fixed points
$p_1$ and $p_2$,  since for every $n\geq 1$ $\langle x_1,x_2\rangle\in X^2$ is
proximal for $T$ iff so is for $T^n$. Then it is easy to see that
for every transitive point $x\in X$, $\langle x,p_1\rangle$ and $\langle
x, p_2\rangle$ are proximal, but $\langle p_1, p_2\rangle$ can not
be proximal.

(2)  Let $S$ and $R$  be the maximal scrambled sets which contain
$\{x, p_1\}$ and $\{x, p_2\}$  respectively. Clearly, $S$ and $R$
are not separable. Next, we show that $S$ and $R$ also are not
equivalent. If there exists a bijection $\phi: S\to R$ such that
$$\lim_{n\to \infty}d(f^n(y)- f^n(\phi(y)))=0\quad \text{for any }y\in S, $$
then $\langle p_1, \phi(p_1)\rangle$ is asymptotic. Thus, $\langle
p_1, p_2 \rangle$ is  proximal since $\langle p_2, \phi(p_1)
\rangle$ is proximal. This is a contradiction.
\end{proof}

Now we state the main result of this paper: there are various
equivalent conditions of  positive entropy which may relate to the
number ``$3$".

\begin{thm}
Let  $f\in C(I,I)$. Then the following conditions are equivalent:
\begin{enumerate}
\item $h(f)>0$;
\item $Prox_f$ is not an equivalence relation;
\item there exist two maximal scrambled sets which are neither equivalent nor separable;
\item there exists some $3$-scrambled tuple;
\item there exists some $3$-sensitive transitive subsystem;
\item there exists some essential $3$-IN-tuple.
\end{enumerate}
\end{thm}
\begin{proof}
(1)$\Rightarrow$(2) By Theorem \ref{interval-postive-entropy-eq1}(2) and Lemma \ref{tran-sys-lemma}(1).

(2)$\Rightarrow$(1) By Theorem \ref{intevla-proximal-equivalence}.

(1)$\Rightarrow$(3) By Theorem \ref{interval-postive-entropy-eq1}(2) and Lemma \ref{tran-sys-lemma}(2).

(3)$\Rightarrow$(1) By Theorem \ref{interval-scr-eq-sep}.

(1)$\Rightarrow$(4) By Theorem \ref{interval-postive-entropy-eq1}(2) and Theorem \ref{Devaney-chaos-thm}(2).

(4)$\Rightarrow$(1) By Theorem \ref{interval-3-srca-tuple}.

(1)$\Rightarrow$(5) By Theorem \ref{interval-postive-entropy-eq1}(2) and Theorem \ref{Devaney-chaos-thm}(1).

(5)$\Rightarrow$(1) If $h(f)=0$, by Theorem
\ref{interval-adding-machine}, any infinite  transitive subsystem is
at most $2$ to $1$ extension of the adding machine, then it is not
$3$-sensitive \cite{SYZ08}.

(1)$\Rightarrow$ (6) In \cite{HY06}, it shows that if $h(f)>0$, then there exists some essential $3$-IN-tuple.

(6)$\Rightarrow$(1) By Theorem \ref{interval-3-seq-ent-pair}.
\end{proof}

\begin{thm}\cite{Block1992}
Let $f\in C(I,I)$. If $f$ is transitive, then exactly one of the following alternatives holds:
\begin{enumerate}
\item $f$ is strongly mixing,
\item there exists a fixed point $c\in (0,1)$ such that $f([0,c])=[c,1]$, $f([c,1])=[0,c]$
and both $f^2|_{[0,c]}$ and $f^2|_{[c,1]}$ are strongly mixing.
\end{enumerate}
\end{thm}
\begin{thm}\cite{Block1992}\label{interval-strong-mixing}
Let $f\in C(I,I)$. Then the following conditions are equivalent:
\begin{enumerate}
\item $f^2$ is transitive;
\item $f$ is weakly mixing;
\item $f$ is strongly mixing;
\item for any $\varepsilon>0$ and non degenerate subinterval $J\subset I$, there exists a $N>0$ such that
$$f^n(J)\supset [\varepsilon, 1-\varepsilon] \text{ for all } n\geq N.$$
\end{enumerate}
\end{thm}

\begin{thm}\label{interval-stro-mix-eq1}
Let $f\in C(I,I)$. Then the following conditions are equivalent:
\begin{enumerate}
\item $f$ is strongly mixing;
\item $\mathcal C(\mathcal U)>2$ for every open cover $\mathcal U$ which consists of
two non dense open sets, where $\mathcal C(\mathcal U)=\lim_{n\to \infty}N(\bigvee_{i=1}^{n-1}f^{-i}(\mathcal U))$;
\item $f$ is uniformly positive entropy;
\item $f$ is topological K system.
\end{enumerate}
\end{thm}
\begin{proof}

(4)$\Rightarrow$(3) and (3)$\Rightarrow$(2) is trivial. 

(2)$\Rightarrow$(1). We first prove a claim which implies that $f$ is transitive.

{\bf Claim}: Let $(X,T)$ be a TDS. If $\mathcal C(\mathcal U)>2$ for
every open cover  $\mathcal U$ of $X$ which consists of two non
dense open sets, then $T$ is transitive.

{\bf Proof of the Claim: } We follow the idea in \cite{BHM00}. If
not, there exist two nonempty open  subsets $U,V$ of $X$ such that $T^nU\cap
V=\emptyset$ for all $n\geq 1$.

Case 1.  If there is some $n\geq 1$ such that $T^nV \cap U \neq
\emptyset$, let $m$ be the minimum of such values. Let $E=V\cap
T^{-m}U$, then $E\cap T^{-n}E=\emptyset$ for all $n\geq 1$. Choose
two closed  subsets $U_1$ and $V_1$  of $E$ and $T^{-1}E$
respectively which have non-empty interior. Let
$\mathcal{U}=\{U_1^c, V_1^c\}$, it is easy to verify that
$\mathcal{C}(\mathcal{U})=2$. This is a contraction.

Case 2. $T^n V\cap U=\emptyset$ for all $n\geq 1$. Choose two closed
subsets $U_1$ and $V_1$   of $U$ and $V$ respectively which have
non-empty interior. Let $\mathcal{U}=\{U_1^c, V_1^c\}$, it is easy
to verify that $\mathcal{C}(\mathcal{U})=2$. This also is a
contraction.

Thus, the proof of the Claim is complete.

So $f$ is transitive. Now assume that $f$ is not strongly mixing, then
there exists a fixed  point $c\in (0,1)$ such that $f([0,c])=[c,1]$,
$f([c,1])=[0,c]$. There exists $\varepsilon>0$ such that
$c+\varepsilon, c-\varepsilon\in (0,1)$. Let $\mathcal U=\{[0,
c+\varepsilon), (c-\varepsilon, 1]\}$. It is easy to verify that
$\mathcal C(\mathcal U)=2$. This also is a contradiction.

(1)$\Rightarrow$(4) To show that $f$ is topological K , it is
sufficient to show that every $k$-tuple of non-empty open  subsets
$(U_1, \ldots, U_k)$ has an independence set of positive density.

Since $f$ is strongly mixing, by Theorem \ref{interval-strong-mixing}(4),
there exist some $n\geq 1$ and nonempty
open subset $V_i\subset U_i$ for $1\leq i\leq k $ such that
$\bigcap_{i=1}^n f^n V_i \supset \bigcup_{i=1}^k V_i$. Then
$n\mathbb{N}=\{n,2n,3n,\ldots\}$ is an independence set for $(U_1,
\ldots, U_k)$ since
$$f^{-n} V_{s(1)}\cap f^{-2n}V_{s(2)}\cap\cdots \cap f^{-mn}V_{s(m)}\neq \emptyset$$
holds for all $m\geq 1$ and all $s\in\{1,2,\ldots k\}^m$.
\end{proof}

\begin{rem}
Recall a TDS $(X,T)$ is called to be of {\em completely positive entropy}
if each of its non-trivial factors has positive entropy.
In \cite{B92}, it showed that there exists a TDS which is of completely
positive entropy but not of uniformly positive entropy.  There are
also some examples for interval maps. For example, let $f(x)= 1/2 +
2x$ on $[0, 1/4]$, $3/2-2x$ on $[1/4, 1/2]$ and $1-x$ on $[1/2, 1]$.
It is easy to check that $f$ is of  completely positive entropy but
not of uniformly positive entropy. Moreover, $f^2$ is also of
completely positive entropy but not transitive.
\end{rem}
\section{Topological null system}

In \cite{HY09}, Huang and Ye introduced the notion of maximal
pattern entropy. For a TDS $(X,T)$, $n \in \mathbb{N}$ and a
finite open cover $\mathcal{U}$, let
$$p^*_{X,\mathcal{U}}(n) = \max_{(t_1<t_2<\cdots<t_n)\in\mathbb{Z}_+^n }N \left(\bigvee^n_{i=1} T^{-t_i}\mathcal{U}\right).$$

The maximal pattern entropy of $T$ with respect to $\mathcal{U}$ is
defined by
$$h^*_{top}(T,\mathcal{U}) = \lim_{n\to\infty} \frac{1}{n}\log p^*_{X,\mathcal{U}}(n).$$

The {\em maximal pattern entropy} of $(X,T)$ is
$$h^*_{top}(T)= \sup_{\mathcal{U}\in \mathcal{C}_X^o} h^*_{top}(T,\mathcal{U}).$$
Then a TDS $(X,T)$ is null iff $h^*_{top}(T)=0$, since $h^*_{top}(T)=\sup_A h_A(T)$.

In \cite{HY09}, Huang and Ye proved that for a null TDS
defined on a zero dimensional space, $p^*_{X,\mathcal{U}}(n)$ is of
polynomial order for each open cove $\mathcal{U}$ of $X$. They also

{\bf Conjecture:} If a TDS $(X,T)$ is null, then it is true that
$p^*_{X,\mathcal{U}}(n)$ is of polynomial order for each open cover
$\mathcal{U}$ of $X$.

In the following section, we prove that the conjecture holds
for interval maps. Before doing this, we need some lemmas. Let $\omega(f)=\bigcup\{\omega_f(x):x\in I\}$.

\begin{lem} \cite{FS91} \label{Smital-lemma} Let $f\in C(I,I)$. If $f$ is null, then for every
$\varepsilon > 0$ there are points $x_1, x_2, \ldots, x_k \in
\omega(f)$ and an open set $U \supset \omega(f)$ with the following property: if
$$f^j(x)\in U \quad for  \quad 0\leq j \leq r ,$$
then there exists some $i\in\{1,2,\ldots, k\}$ such that for any $j$ with $0\leq j \leq r$,
$$|f^j(x)-f^j(x_i)|<\varepsilon.$$
\end{lem}

\begin{lem}\label{Sarkovskii-lemma} {\rm (\cite{Sh67} or \cite[Corollary IV.13]{Block1992})}
Let $f\in C(I,I)$. Then for any neighborhood $U$ of
$\omega(f)$ there is an integer $q>0$ such that the number of points
of an arbitrary trajectory lying outside $U$ is less than $q$.
\end{lem}

\begin{thm} \label{interval-poloynomial-order}
Let $f\in C(I,I)$. If $f$ is null, then $p^*_{I,\mathcal{U}}(n)$ is of polynomial order for each open
cover $\mathcal{U}$ of $I$.
\end{thm}

\begin{proof}
We follow the idea in \cite{FS91}. Let $\mathcal{U}$ be an open cover of $I$ with
Lebesgue number $\delta$ and $n \in \mathbb{N}$. For any
$\vec{t}=(t_1<t_2<\cdots<t_n)\in\mathbb{Z}_+^n$, it is
well known that
$$N\left(\bigvee^n_{i=1}T^{-t_i}\mathcal{U}\right)\leq S\left(\vec{t},f,\frac{\delta}{2}\right),$$
where $S\left(\vec{t},f,\frac{\delta}{2}\right)$ is the minimal
cardinality of $\left(\vec{t},f,\frac{\delta}{2}\right)$-spanning sets.
Recall that a set $E \subset I$ is called a $\left(\vec{t},f,\varepsilon\right)$-spanning set,
if for any $x \in I$, there exists some $y \in E$ such that
$|f^{t_i}(x)-f^{t_i}(y)|<\varepsilon$ for $1 \leq i \leq n$. Let
$\varepsilon=\frac{\delta}{4}$ and $U$ and $x_1,\ldots,x_k$ be as in
Lemma \ref{Smital-lemma}. Let $\{K_i\}_{i=1}^s$ be pairwise disjoint
set with diam$(K_i)<\varepsilon$ for any $i$, and $K_1\cup\cdots\cup
K_s=I\backslash U$. Assign to any $x\in I$ an itinerary
$\alpha_{\vec{t}}(x)=\{\alpha_{t_i}(x)\}_{i=1}^n$ such
that $\alpha_{t_i}(x)=K_j$ if $f^{t_i}(x)\in K_j$. If $f^{t_i}(x)\in
U$, let $M(t_i)$ be the maximal subinterval of the set of nonnegative
integers such that $t_i\in M(t_i)$ and $f^k(x)\in U $, for all $k
\in M(t_i)$. Then by Lemma \ref{Smital-lemma}, there exists some $r\in\{1,\ldots,k\}$
such that $|f^k(x)-f^k(x_r)|<\varepsilon<\frac{\delta}{2}$ for any
$k \in M(t_i)$. Put $\alpha_{t_j}(x)=r$ for any $t_j\in M(t_i)$.

It is easy to see that for any $x,y \in I$,
$\alpha_{\vec{t}}\,(x)=\alpha_{\vec{t}}\,(y)$
implies $|f^{t_i}(x)-f^{t_i}(y)|<2\varepsilon=\frac{\delta}{2}$ for
all $1\leq i \leq n$. So we get
$$S\left(\vec{t},f,\frac{\delta}{2}\right)\leq C\left(\vec{t}\,\right),$$
where $C\left(\vec{t}\,\right)$ is the number of all the possible
codes $\alpha_{\vec{t}}\,(x)$. By Lemma \ref{Sarkovskii-lemma},
there exists an integer $q>0$ such that the number of points of  an
arbitrary trajectory lying outside $U$ is less than $q$.
Consequently, every code $\alpha_{\vec{t}}\,(x)$ consists of at most
$2q+1$ blocks, and each block is formed by only one of the symbols
$1,\ldots,k,K_1,\ldots,K_s$ (with possible repetitions). Therefore,
$$N \left(\bigvee^n_{i=1} T^{-t_i}\mathcal{U}\right)\leq C\left(\vec{t}\,\right)\leq (k+s)^{2q+1}n^q.$$
We remark that the choice of $k,s$ and $q$ depend only on $\mathcal{U}$ but not $\vec{t}$. Thus,
$$p^*_{I,\mathcal{U}}(n) \leq (k+s)^{2q+1}n^q$$
for all $n\geq 1$
\end{proof}

\begin{rem}
Recall that a space $X$ is called a {\em tree} if it is a connected space
that is a union of finite  number of intervals, but does not contain
a subset homeomorphic to a circle. If one is acquainted  with the
dynamical properties of the tree maps, it is not hard to see that
Theorem  \ref{interval-stro-mix-eq1} and  Theorem
\ref{interval-poloynomial-order} hold for tree maps, but in
Theorem \ref{interval-stro-mix-eq1}(2), the number ``$\,2$" should
be replaced by a sufficient large integer which is associated with
the number of endpoints of the tree.
\end{rem}
\subsection*{Acknowledgement}
The author would like to thank Prof. Xiangdong Ye
for the careful reading and helpful suggestions.
This work was supported in part by the National Natural Science Foundation of China (Nos.\@ 11001071, 11071231).


\end{document}